\documentclass[11pt,a4paper,reqno]{amsart}

\usepackage{amsrefs}

\usepackage{graphicx}

\usepackage{amsfonts,amssymb,mathtools,stmaryrd}

\usepackage[hidelinks]{hyperref}
\usepackage[capitalise,noabbrev]{cleveref}

\newtheorem{theorem}{Theorem}
\newtheorem{lemma}{Lemma}
\newtheorem{proposition}{Proposition}
\newtheorem{definition}{Definition}

\renewcommand{\leq}{\leqslant}
\renewcommand{\geq}{\geqslant}

\makeatletter
\newcommand{\addresseshere}{
	\enddoc@text
	\let\enddoc@text\relax
}
\makeatother

\usepackage{enumitem}
\setlist{leftmargin=*}

\renewcommand{\uppercasenonmath}[1]{}

\setenumerate[1]{label={\arabic*.}}

\title{Bounded Automata Groups are co-{ET0L}}

\author{Alex Bishop}
\address{University of Technology Sydney, Australia}
\urladdr{\url{https://alexbishop.github.io}}
\email{\href{mailto:alexander.bishop@uts.edu.au}{alexander.bishop@uts.edu.au}}

\author{Murray Elder}
\address{University of Technology Sydney, Australia}
\urladdr{\url{https://sites.google.com/site/melderau}}
\email{\href{mailto:murray.elder@uts.edu.au}{murray.elder@uts.edu.au}}

\thanks{Research supported by Australian Research Council grant DP160100486 and an Australian Government Research Training Program Scholarship}

\begin{document}
	
\begin{abstract}
	Holt and R\"over proved that finitely generated bounded automata groups have indexed co-word problem.
	Here we sharpen this result to show they are in fact co-ET0L.
	\bigskip
	
	
\end{abstract}
\maketitle

\section{Introduction}
A recurrent theme in group theory is to understand and classify group-theoretic problems in terms of their formal language complexity \cites{anisimov1971,elder2008,epstein1992,gilman1996,muller1985}.
Many authors have considered the groups whose non-trivial elements, i.e.\@ \emph{co-word problem}, can be described as a context-free language \cites{bleak2016,holt2005,koenig2016,lehnert2007}.
Holt and R\"over went beyond context-free to show that a large class  known as \emph{bounded automata groups} have an indexed  co-word problem  \cite{holt2006}.
This class includes important examples such as Grigorchuk's group of intermediate growth, the Gupta-Sidki groups, and many more \cites{grigorchuk1980,gupta1983,nekrashevych2005,sidki2000}.
For the specific case of the Grigorchuk group, Ciobanu \textit{et al.\@}  \cite{ciobanu2018} showed that the co-word problem was in fact ET0L.
ET0L is a class of languages coming from L-systems which lies strictly between context-free and indexed \cites{lindenmayer1968,lindenmayer1968a,rozenberg1997,rozenberg1973}.
Ciobanu \textit{et al.\@} rely on the grammar description of ET0L for their result. Here we are able to show that all finitely generated bounded automata groups have ET0L co-word problem by instead making use of an equivalent machine description: check-stack pushdown (cspd) automata.

ET0L languages, in particular their deterministic versions, have recently come to prominence in describing solution sets to equations in groups and monoids \cites{ciobanu2016a,diekert2017,diekert2016}.
The present paper builds on the recent resurgence of interest in this class of languages, and demonstrates the usefulness of a previously overlooked machine description.

For a group $G$ with finite generating set $X$, we denote by $\mathrm{coW}(G,X)$ the set of all words in the language $(X \cup X^{-1})^\star$ that represent non-trivial elements in $G$.
We call $\mathrm{coW}(G,X)$ the {co-word problem} for $G$ (with respect to $X$).
Given a class $\mathcal{C}$ of formal languages that is closed under inverse homomorphism, if $\mathrm{coW}(G,X)$ is in $\mathcal{C}$ then so is $\mathrm{coW}(G,Y)$ for any finite generating set $Y$ of $G$ \cite{holt2005}.
Thus, we say that a group is \emph{co-$\mathcal{C}$} if it has a co-word problem in the class $\mathcal{C}$.
Note that ET0L  is a full AFL \cite{culik1974} and so is closed under inverse homomorphism.

This paper is organised as follows.
In \cref{sec:language-definition} 
we define ET0L languages and cspd automata,
in \cref{sec:bounded autom} we define bounded automata groups, and in \cref{sec:mainthm} we give a constructive proof that such groups are co-ET0L.
Further, in \cref{sec:machine equivalence} we provide a self-contained proof of the equivalence of ET0L languages and languages accepted by cspd automata.

\section{ET0L Languages and CSPD Automata}
\label{sec:language-definition}

An \emph{alphabet} is a finite set.
Let $\Sigma$ and $V$ be two alphabets which we will call the \emph{terminals} and \emph{non-terminals}, respectively.
We will use lower case letters to represent terminals in $\Sigma$ and upper case letters for non-terminals in $V$.
By $\Sigma^\star$, we will denote the set of words over $\Sigma$ with $\varepsilon \in \Sigma^\star$ denoting the empty word.

A \emph{table}, $\tau$, is a finite set of context-free replacement rules where each non-terminal, $X \in V$, has at least one replacement in $\tau$.
For example, with $\Sigma = \{a,b\}$ and $V = \{S,A,B\}$, the following are tables.
\begin{equation}
\label{eq:alt example tables}
	\alpha \, : \,
		\left\{
		\begin{aligned}
			S &\to SS \ \vert \ S \ \vert \ AB\\
			A &\to A\\
			B &\to B
		\end{aligned}
		\right.
	\qquad
	\beta \, : \,
		\left\{
		\begin{aligned}
			S &\to S\\
			A &\to aA\\
			B &\to bB
		\end{aligned}
		\right.
	\qquad
	\gamma \, : \,
		\left\{
		\begin{aligned}
			S &\to S \\
			A &\to \varepsilon \\
			B &\to \varepsilon
		\end{aligned}
		\right.
\end{equation}

We apply a table, $\tau$, to the word  $w \in (\Sigma \cup V)^\star$ to obtain a word $w'$, written $w \to^\tau w'$, by performing a replacement in $\tau$ to each non-terminal in $w$.
If a table includes more than one rule for some non-terminal, we nondeterministically apply any such rule to each occurrence.
For example, with $w = SSSS$ and $\alpha$ as in (\ref{eq:alt example tables}), we can apply $\alpha$ to $w$ to obtain $w' = SABSSAB$.
Given a sequence of tables $\tau_1$, $\tau_2$, \dots, $\tau_k$, we will write $w \to^{\tau_1 \tau_2 \cdots \tau_k} w'$ if there is a sequence of words $w = w_1$, $w_2$, \dots, $w_{k+1} = w'$ such that $w_{j} \to^{\tau_j} w_{j+1}$ for each $j$.
Notice here that the tables are applied from left to right.

\begin{definition}[Asveld \cite{asveld1977}]
	\label{def:et0l grammar}
	An \emph{ET0L grammar} is a 5-tuple $G = (\Sigma, V, T, \mathcal{R}, S)$, where
	\begin{enumerate}
		\item $\Sigma$ and $V$ are the alphabets of \emph{terminals} and \emph{non-terminals}, respectively;
		\item $T = \{\tau_1, \tau_2 ,\dots, \tau_k\}$ is a finite set of \emph{tables};
		\item $\mathcal{R} \subseteq T^\star$ is a regular language called the \emph{rational control}; and
		\item $S \in V$ is the \emph{start symbol}.
	\end{enumerate}
	The \emph{ET0L language} produced by the grammar $G$, denoted $L(G)$, is 
	\[
		L(G)
		\coloneqq
		\left\{
			w \in \Sigma^\star
		\, :\, 
			S \to^{v} w
			\ \ 
			\text{for some}
			\ \ 
			v \in \mathcal{R}
		\right\}.
	\]
\end{definition}
For example, with $\alpha$, $\beta$ and $\gamma$ as in (\ref{eq:alt example tables}), the language produced by the above grammar with rational control $
	\mathcal{R}
	=
	\alpha^\star \beta^\star \gamma
$ is $
	\{
		( a^n b^n )^m
		\, : \,
		n,m \in \mathbb{N}
	\}
$.

\subsection{CSPD Automata}
\label{subsec:cspd_machines}

A \emph{cspd automaton}, introduced in \cite{leeuwen1976}, is a nondeterministic finite-state automaton with a one-way input tape, and access to both a \emph{check-stack} (with stack alphabet $\Delta$) and a pushdown stack (with stack alphabet $\Gamma$), where access to these two stacks is tied in a very particular way.
The execution of a cspd machine is separated into two stages.

In the first stage the machine is allowed to push to its check-stack but \textit{not} its pushdown, and further, the machine will not be allowed to read from its input tape.
Thus, the set of all possible check-stacks that can be constructed in this stage forms a regular language which we will denote as $\mathcal{R}$.

In the second stage, the machine can no longer alter its check-stack, but is allowed to access its pushdown and input tape.
We  restrict the machine's access to its stacks so that it can only move along its check-stack by pushing and popping items to and from its pushdown.
In particular, every time the machine pushes a value onto the pushdown it will move up the check-stack, and every time it pops a value off of the pushdown it will move down the check-stack; see \cref{fig:cspd stage 2} for an example of this behaviour.

\begin{figure}[!ht]
	\centering
	
	\begin{minipage}{0.3\linewidth}
		\centering
		\includegraphics[width=\linewidth]{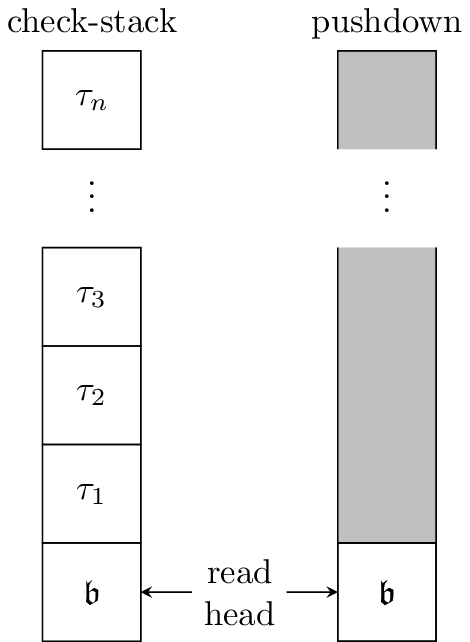}
	\end{minipage}
	~
	\begin{minipage}{0.3\linewidth}
		\centering
		\includegraphics[width=\linewidth]{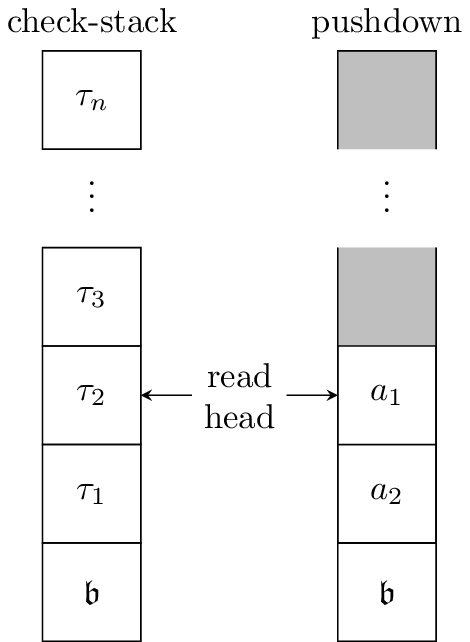}
	\end{minipage}
	~
	\begin{minipage}{0.3\linewidth}
		\centering
		\includegraphics[width=\linewidth]{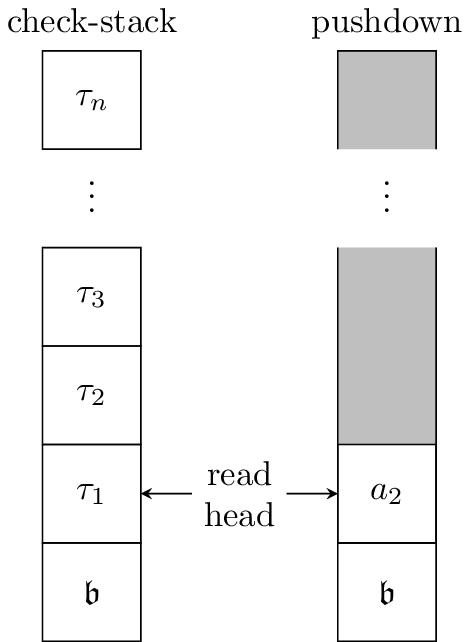}
	\end{minipage}
	
	\caption{An example of a cspd machine pushing $w = a_1 a_2$, where $a_1,a_2 \in \Delta$, onto its pushdown stack, then popping $a_1$.}
	\label{fig:cspd stage 2}
	
\end{figure}

We define a cspd machine formally as follows.

\begin{definition}
	\label{def:cspd machine}
	A \emph{cspd machine} is a 9-tuple $\mathcal{M} = (Q, \Sigma, \Gamma, \Delta, \mathfrak b, \mathcal{R}, \theta, q_0, F)$, where
	\begin{enumerate}
		\item $Q$ is the set of \emph{states};
		
		\item $\Sigma$ is the \emph{input alphabet};
		
		\item $\Gamma$ is the \emph{alphabet for the pushdown};
		
		\item $\Delta$ is the \emph{alphabet for the check-stack};
		
		\item $\mathfrak b \notin \Delta \cup \Gamma$ is the \emph{bottom of stack symbol};
		
		\item $\mathcal{R} \subseteq \Delta^\star$ is a \emph{regular language of allowed check-stacks};
		
		\item $\theta$ is a finite subset of
		\[
		(
				Q
			\times
				( \Sigma \cup \{\varepsilon\} )
			\times
				(
					(\Delta \times \Gamma)
					\cup
					\{
						(\varepsilon,\varepsilon),
						(\mathfrak{b},\mathfrak{b})
					\}
				)
		)
		\times
			(Q \times ( \Gamma \cup \{\mathfrak{b}\})^\star ),
		\]
		called the \emph{transition relation} (see below for allowable elements of $\theta$);
		
		\item $q_0 \in Q$ is the \emph{start state}; and
		
		\item $F \subseteq Q$ is the set of \emph{accepting states}.
	\end{enumerate}
\end{definition}

In its initial configuration, the machine will be in state $q_0$, the check-stack will contain $\mathfrak{b}w$ for some nondeterministic choice of $w \in \mathcal{R}$, the pushdown will contain only the letter $\mathfrak{b}$, the read-head for the input tape will be at its first letter, and the read-head for the machine's stacks will be pointing to the $\mathfrak{b}$ on both stacks.
From here, the machine will follow transitions as specified by $\theta$, each such transition having one of the following three forms, where $a\in \Sigma\cup\{\varepsilon\}$, $p,q\in Q$ and $w \in \Gamma^\star$.

\begin{enumerate}

\item
$((p,a,(\mathfrak{b},\mathfrak{b})),(q, w\mathfrak{b})) \in \theta$ meaning that if the machine is in state $p$, sees $\mathfrak{b}$ on both stacks and is able to consume $a$ from its input;
then it can follow this transition to  consume $a$, push $w$ onto the pushdown and move to state $q$.

\item
$((p,a,(d,g)),(q,w)) \in \theta$ where $ (d,g) \in \Delta \times \Gamma$,
meaning that if the machine is in state $p$, sees $d$ on its check-stack, $g$ on its pushdown, and is able to consume $a$ from its input;
then it can follow this transition to  consume $a$, pop $g$, then push $w$ and move to state $q$.

\item
$((p,a,(\varepsilon,\varepsilon)),(q,w)) \in \theta$
meaning that if the machine is in state $p$ and can consume $a$ from its input;
then it can follow this transition to  consume $a$, push $w$ and move to state $q$.
\end{enumerate}

In the previous three cases, $a = \varepsilon$ corresponds to a transition in which the machine does not consume a letter from input.
We use the convention that, if $w = w_1 w_2 \cdots w_k$ with each $w_j \in \Gamma$, then the machine will first push the $w_k$, followed by the $w_{k-1}$ and so forth. 
The machine accepts if it has consumed all its input and is in an accepting state $q \in F$.

In \cite{leeuwen1976} van Leeuwen proved that the class of languages that are recognisable by cspd automata is precisely the class of ET0L languages. 
For a self-contained proof of this fact, see \cref{sec:machine equivalence}.

\section{Bounded Automata Groups}
\label{sec:bounded autom}

For $d \geq 2$, let $\mathcal{T}_d$ denote the $d$-regular rooted tree, that is, the infinite rooted tree where each vertex has exactly $d$ children.
We identify the vertices of $\mathcal{T}_d$ with words in $\Sigma^\star$ where $\Sigma = \{ a_1, a_2, \ldots, a_d \}$.
In particular, we will identify the root with the empty word $\varepsilon \in \Sigma^\star$ and, for each vertex $v \in \mathrm{V}(\mathcal{T}_d)$, we will identify the $k$-th child of $v$ with the word $v a_k$, see \cref{fig:tree-vertex-labelling}.
\begin{figure}[!ht]
	\centering
	\includegraphics[width=0.45\linewidth]{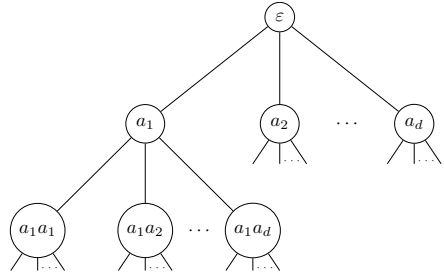}
	\caption{A labelling of the vertices of $\mathcal{T}_d$.}
	\label{fig:tree-vertex-labelling}
\end{figure}

Recall that an automorphism of a graph is a bijective mapping of the vertex set that preserves adjacencies.
Thus, an automorphism of $\mathcal{T}_d$ preserves the root and ``levels" of the tree.
The set of all automorphisms of $\mathcal{T}_d$ is a group, which we denote by $\mathrm{Aut}(\mathcal{T}_d)$.
We denote the \emph{permutation group of $\Sigma$} as $\mathrm{Sym}(\Sigma)$.
An important observation  is that $\mathrm{Aut}(\mathcal{T}_d)$ can be seen as the wreath product $\mathrm{Aut}(\mathcal{T}_d) \wr \mathrm{Sym}(\Sigma)$, since any automorphism $\alpha \in \mathrm{Aut}(\mathcal{T}_d)$ can be written uniquely as $\alpha = (\alpha_1, \alpha_2, \ldots, \alpha_d) \cdot \sigma$ where $\alpha_i \in \mathrm{Aut}(\mathcal{T}_d)$ is an automorphism of the sub-tree with root $a_i$,  and $\sigma \in \mathrm{Sym}(\Sigma)$ is a permutation  of the first level.

Let $\alpha \in \mathrm{Aut}(\mathcal{T}_d)$ where $\alpha = (\alpha_1, \alpha_2, \ldots, \alpha_d) \cdot \sigma \in \mathrm{Aut}(\mathcal{T}_d) \wr \mathrm{Sym}(\Sigma)$.
For any $b = a_i \in \Sigma$, the \emph{restriction of $\alpha$ to $b$}, denoted $\left.\alpha\right\vert_b \coloneqq \alpha_i$, is the action of $\alpha$ on the sub-tree with root $b$.
Given any vertex $w = w_1 w_2 \cdots w_k \in \Sigma^\star$ of $\mathcal{T}_d$, we can define the \emph{restriction of $\alpha$ to $w$} recursively as
\[
	\left.\alpha\right\vert_w
	=
	\left.
		\left(\left.\alpha\right\vert_{w_1w_2\cdots w_{k-1}}\right)
	\right\vert_{w_k}
\]
and thus describe the action of $\alpha$ on the sub-tree with root $w$.

A \emph{$\Sigma$-automaton}, $(\Gamma,v)$, is a finite directed graph with a distinguished vertex $v$, called the \emph{initial state}, and a $(\Sigma\times\Sigma)$-labelling of its edges, such that each vertex has exactly $\left\vert\Sigma\right\vert$ outgoing edges:
with one outgoing edge with a label of the form $(a, \cdot)$ and one outgoing edge with a label of the form $( \cdot, a)$ for each $a \in \Sigma$.
Thus, the outgoing edges define a permutation of $\Sigma$.

Given some $\Sigma$-automaton $(\Gamma,v)$, where $\Sigma = \{a_1,\ldots,a_d\}$, we can define an automorphism $\alpha_{(\Gamma,v)} \in \mathrm{Aut}(\mathcal{T}_d)$ as follows.
For any given vertex $b_1 b_2 \cdots b_k \in \Sigma^\star = \mathrm{V}\!\left(\mathcal{T}_d\right)$, there exists a unique path in $\Gamma$ starting from the initial vertex, $v$, of the form
$
	(b_1, b_1')
	\,
	(b_2, b_2')
	\,
	\cdots
	\,
	(b_k, b_k')
$,
thus we will now define $\alpha_{(\Gamma,v)}$ such that $\alpha_{(\Gamma,v)} (b_1 b_2 \cdots b_k) = b_1' b_2' \cdots b_k'$.
Notice that it follows from the definition of a $\Sigma$-automaton that $\alpha_{(\Gamma,v)}$ is a tree automorphism as required.
For an example of a $\Sigma$-automaton, see \cref{fig:sigma_autom_girgorchuk}.

\begin{figure}[!ht]
	\centering
	\includegraphics[width=.7\linewidth]{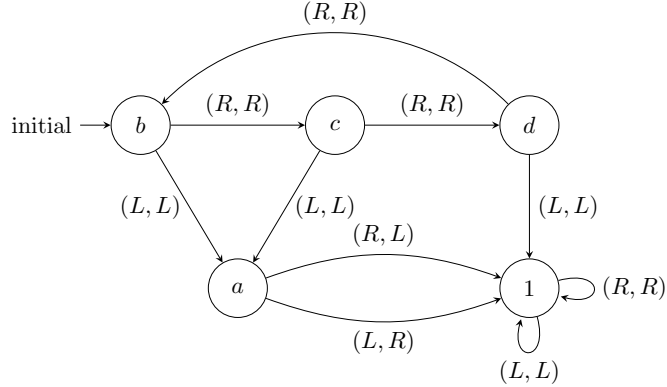}
	\caption{A $\Sigma$-automaton for $b$ in Grigorchuk's group.}
	\label{fig:sigma_autom_girgorchuk}
\end{figure}

An \emph{automaton automorphism}, $\alpha$, of the tree $\mathcal{T}_d$ is an automorphism for which there exists a $\Sigma$-automaton, $(\Gamma,v)$, such that $\alpha = \alpha_{(\Gamma,v)}$.
The set of all automaton automorphisms of the tree $\mathcal{T}_d$ form a group which we will denote as $\mathcal{A}\!\left(\mathcal{T}_d\right)$.
A subgroup of $\mathcal{A}(\mathcal{T}_d)$ is called an \emph{automata group}.

An automorphism $\alpha \in \mathrm{Aut}(\mathcal{T}_d)$ will be called \emph{bounded} (originally defined in \cite{sidki2000}) if there exists a constant $N \in \mathbb{N}$ such that for each $k\in \mathbb{N}$, there are no more than $N$ vertices $v\in \Sigma^\star$ with $\left\vert v \right\vert = k$ (i.e.\@ at level $k$) such that $\left.\alpha\right\vert_v \neq 1$.
Thus, the action of such a bounded automorphism will, on each level, be trivial on all but (up to) $N$ sub-trees.
The set of all such automorphisms form a group which we will denote as $\mathcal{B}(\mathcal{T}_d)$.
The group of all \emph{bounded automaton automorphisms} is defined as the intersection $\mathcal{A}(\mathcal{T}_d) \cap \mathcal{B}(\mathcal{T}_d)$, which we will denote as $\mathcal{D}(\mathcal{T}_d)$.
A subgroup of $\mathcal{D}(\mathcal{T}_d)$ is called a  \emph{bounded automata group}.

A \emph{finitary automorphism} of $\mathcal{T}_d$ is an automorphism $\phi$ such that there exists a constant $k \in \mathbb{N}$ for which $\left. \phi \right\vert_v = 1$ for each $v \in \Sigma^\star$ with $\left\vert v \right\vert = k$.
Thus, a finitary automorphism is one that is trivial after some $k$ levels of the tree.
Given a finitary automorphism $\phi$, the smallest $k$ for which this definition holds will be called its \emph{depth} and will be denoted as $\mathrm{depth}(\phi)$.
We will denote the group formed by all finitary automorphisms of $\mathcal{T}_d$ as $\mathrm{Fin}(\mathcal{T}_d)$.
See \cref{fig:finitary examples} for examples of the actions of finitary automorphisms on their associated trees (where any unspecified sub-tree is fixed by the action).

\begin{figure}[!ht]
	\centering
	\begin{minipage}{.3\linewidth}
		\centering
		\includegraphics[height=1.5cm]{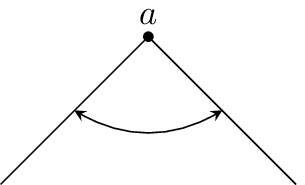}
	\end{minipage}
	~
	\begin{minipage}{.3\linewidth}
		\centering
		\includegraphics[height=1.5cm]{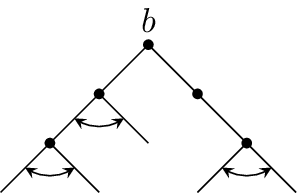}
	\end{minipage}
	\caption{Examples of finitary automorphisms $a,b\in\mathrm{Fin}\!\left(\mathcal{T}_2\right)$.}
	\label{fig:finitary examples}
\end{figure}

Let $\delta \in \mathcal{A}\!\left(\mathcal{T}_d\right) \setminus \mathrm{Fin}\!\left(\mathcal{T}_d\right)$. We call $\delta$ a \emph{directed automaton automorphism} if
\begin{equation}
\label{eq:directed automorphism definition}
	\delta
	=
	(\phi_1, \phi_2, \ldots, \phi_{k-1}, \delta', \phi_{k+1}, \ldots, \phi_d) \cdot \sigma
	\in
	\mathrm{Aut}\!\left(\mathcal{T}_d\right) \wr \mathrm{Sym}(\Sigma)
\end{equation}
where each $\phi_j$ is finitary and $\delta'$ is also directed automaton (that is, not finitary and can also be written in this form).
We call $\mathrm{dir}(\delta) = b = a_k \in \Sigma$, where $\delta'=\delta\vert_{b}$ is directed automaton, the \emph{direction} of $\delta$; and we define the \emph{spine} of $\delta$, denoted $\mathrm{spine}(\delta) \in \Sigma^\omega$, recursively such that $\mathrm{spine}(\delta) = \mathrm{dir}(\delta) \, \mathrm{spine}(\delta')$.
We will denote the set of all directed automaton automorphisms as $\mathrm{Dir}(\mathcal{T}_d)$.
See \cref{fig:directed examples} for examples of directed automaton automorphisms (in which $a$ and $b$ are the finitary automorphisms in \cref{fig:finitary examples}).

\begin{figure}[!ht]
	\centering
		\begin{minipage}{.3\linewidth}
			\centering
			\includegraphics[height=3cm]{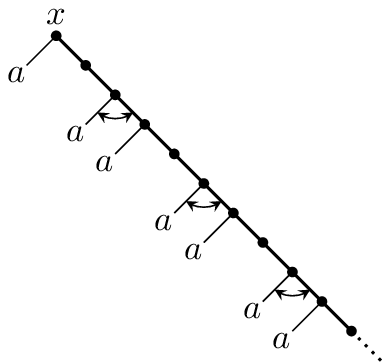}
		\end{minipage}
	~
		\begin{minipage}{.3\linewidth}
			\centering
			\includegraphics[height=3cm]{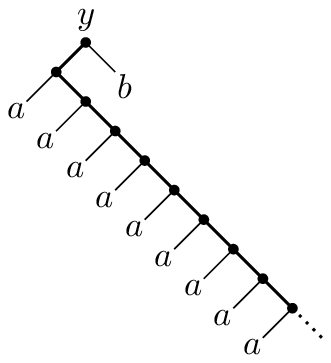}
		\end{minipage}
	~
		\begin{minipage}{.3\linewidth}
			\centering
			\includegraphics[height=3cm]{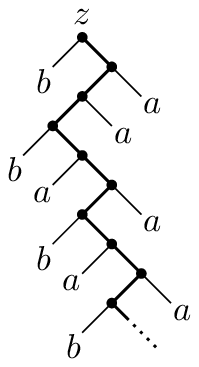}
		\end{minipage}
	\caption{Examples of directed automorphisms $x,y,z \in \mathrm{Dir}(\mathcal{T}_2)$.}
	\label{fig:directed examples}
\end{figure}

The following lemma is essential to prove our main theorem.

\begin{lemma}
\label{lemma:spine is eventualy periodic}
	The spine, $\mathrm{spine}(\delta) \in \Sigma^\omega$, of a directed automaton automorphism, $\delta \in \mathrm{Dir}(\mathcal{T}_d)$, is eventually periodic, that is, there exists some $\iota = \iota_1 \iota_2 \cdots \iota_s \in \Sigma^\star$, called the \emph{initial section}, and
	$\pi = \pi_1 \pi_2 \cdots \pi_t \in \Sigma^\star$, called the \emph{periodic section}, such that $\mathrm{spine}(\delta) = \iota \, \pi^\omega$; and
	\begin{equation}
	\label{eq:restrictions along periodic section}
		\left.
			\delta
		\right\vert_{\iota \, \pi^k \, \pi_1 \pi_2 \cdots \pi_j}
		=
		\left.
			\delta
		\right\vert_{\iota \, \pi_1 \pi_2 \cdots \pi_j}
	\end{equation}
	for each $k,j \in \mathbb{N}$ with $0\leq j <t$.
\end{lemma}

\begin{proof}
Let $(\Gamma,v)$ be a $\Sigma$-automaton such that $\delta = \alpha_{(\Gamma,v)}$.
By the definition of $\Sigma$-automata, for any given vertex $w = w_1 w_2 \cdots w_k \in \Sigma^\star$ of $\mathcal{T}_d$ there exists a vertex $v_w \in \mathrm{V}(\Gamma)$ such that $\delta\vert_w = \alpha_{(\Gamma,v_w)}$.
In particular, such a vertex $v_w$ can be obtained by following the path with edges labelled
$
	(w_1, w_1')
	(w_2, w_2')
	\cdots
	(w_k, w_k')
$.
Then, since there are only finitely many vertices in $\Gamma$, the set of all restrictions of $\delta$ is finite, that is,
\begin{equation}
\label{eq:finitely many restrictions}
	\left\vert
	\left\{
		\left.\delta\right\vert_w
		=
		\alpha_{(\Gamma,v_w)}
		\, : \,
		w \in \Sigma^\star
	\right\}
	\right\vert
	<
	\infty.
\end{equation}
Let $b = b_1 b_2 b_3 \cdots = \mathrm{spine}(\delta) \in \Sigma^\omega$ denote the spine of $\delta$.
Then, there exists some $n,m \in \mathbb{N}$ with $n < m$ such that
\begin{equation}
\label{eq:equivalent restrictions}
	\delta\vert_{b_1 b_2 \cdots b_n}
	=
	\delta\vert_{b_1 b_2 \cdots b_n \cdots b_{m}}
\end{equation}
as otherwise there would be infinitely many distinct restrictions of the form $\delta\vert_{b_1 b_2 \cdots b_k}$ thus contradicting (\ref{eq:finitely many restrictions}).
By the definition spine, it follows that
\[
	\mathrm{spine}
	\left(
	\delta\vert_{b_1 b_2 \cdots b_n}
	\right)
	=
	(b_{n+1}b_{n+2} \cdots b_m)
	\ 
	\mathrm{spine}
	\left(
	\delta\vert_{b_1 b_2 \cdots b_n \cdots b_{m}}
	\right).
\]
and hence, by (\ref{eq:equivalent restrictions}),
\[
	\mathrm{spine}
	\left(
	\delta\vert_{b_1 b_2 \cdots b_n}
	\right)
	=
	(b_{n+1}b_{n+2} \cdots b_m)^\omega.
\]
Thus,
\begin{align*}
	\mathrm{spine}(\delta)
	&=
	(b_1 b_2 \cdots b_n)
	\ 
	\mathrm{spine}
	\left(
	\delta\vert_{b_1 b_2 \cdots b_n}
	\right)
	\\&=
	(b_1 b_2 \cdots b_n)
	\ 
	(b_{n+1}b_{n+2} \cdots b_m)^\omega.
\end{align*}
By taking $\iota = b_1 b_2 \cdots b_n$ and $\pi = b_{n+1} b_{n+2} \cdots b_m$, we have $\mathrm{spine}(\delta) = \iota\,\pi^\omega$.
Moreover, from (\ref{eq:equivalent restrictions}), we have equation (\ref{eq:restrictions along periodic section}) as required.
\end{proof}

Notice that each finitary and directed automata automorphism is also bounded, in fact, we have the following proposition which shows that the generators of any given bounded automata group can be written as words in $\mathrm{Fin}\!\left(\mathcal{T}_d\right)$ and $\mathrm{Dir}\!\left(\mathcal{T}_d\right)$.

\begin{proposition}[Proposition 16 in \cite{sidki2000}]
	\label{prop:bounded automata group if directed}
	The group $\mathcal{D}\!\left(\mathcal{T}_d\right)$ of bounded automata automorphisms is generated by $\mathrm{Fin}\!\left(\mathcal{T}_d\right)$ together with $\mathrm{Dir}\!\left(\mathcal{T}_d\right)$.
\end{proposition}

\section{Main Theorem}
\label{sec:mainthm}

\begin{theorem}
	\label{thm:bounded automata is ET0L}
	Every finitely generated bounded automata group is co-ET0L.
\end{theorem}
The idea of the proof is straightforward:
we construct a cspd machine that chooses a vertex $v \in \mathrm{V}(\mathcal{T}_d)$, writing its labels on the check-stack and a copy on its pushdown;
as it reads letters from input, it uses the pushdown to keep track of where the chosen vertex is moved;
and finally it checks that the pushdown and check-stack differ.
The full details are as follows.

\begin{proof}

	Let $G \subseteq \mathcal{D}(\mathcal{T}_d)$ be a bounded automata group with finite symmetric generating set $X$.
	By \cref{prop:bounded automata group if directed}, we can define a map
	\[
		\varphi:
		X
		\to
		\left(
			\mathrm{Fin}(\mathcal{T}_d) \cup \mathrm{Dir}(\mathcal{T}_d)
		\right)^\star
	\] so that $x =_{\mathcal{D}(\mathcal{T}_d)} \varphi(x)$ for each $x \in X$.
	Let
	\[
		Y
		=
		\left\{
			\alpha \in \mathrm{Fin}(\mathcal{T}_d) \cup \mathrm{Dir}(\mathcal{T}_d) 
			\,:\,
			\alpha\text{ or } \alpha^{-1} \text{ is a factor of } \varphi(x) \text{ for some } x \in X
		\right\}
	\] which is finite and symmetric.
	Consider the group $H \subseteq \mathcal{D}(\mathcal{T}_d)$ generated by $Y$.
	Since ET0L is closed under inverse word homomorphism, it suffices to prove that $\mathrm{coW}(H,Y)$ is ET0L, as $\mathrm{coW}(G,X)$ is its inverse image under the mapping $X^\star \to Y^\star$ induced by  $\varphi$.
	We construct a cspd machine $\mathcal{M}$ that recognises $\mathrm{coW}(H,Y)$, thus proving that $G$ is co-ET0L.
	
	Let $\alpha = \alpha_1 \alpha_2 \cdots \alpha_n \in Y^\star$ denote an input word given to $\mathcal{M}$.
	The execution of the cspd will be separated into four stages;
	(1) choosing a vertex $v \in \Sigma^\star$ of $\mathcal{T}_d$ which witnesses the non-triviality of $\alpha$ (and placing it on the stacks);
	(2a)  reading a finitary automorphism from the input tape;
	(2b)  reading a directed automaton automorphism from the input tape; and
	(3)  checking that the action of $\alpha$ on $v$ that it has computed is non-trivial.
	
	After Stage~1, $\mathcal{M}$ will be in state $q_\mathrm{comp}$.
	From here, $\mathcal{M}$ nondeterministically decides to either read from its input tape, performing either Stage~2a or 2b and returning to state $q_\mathrm{comp}$; or to finish reading from input by performing Stage~3.
	
	We	set both the check-stack and pushdown alphabets to be $\Sigma \sqcup \{\mathfrak{t}\}$,
	i.e., we have $\Delta=\Gamma=\Sigma \sqcup \{\mathfrak{t}\}$.
	The letter $\mathfrak{t}$ will represent the top of the check-stack.
	
	\vspace*{.8em}\noindent
	{\itshape Stage 1: choosing a witness $v = v_1 v_2 \cdots v_m \in \Sigma^\star$.}
	\vspace*{.25em}
	
	\noindent
	If $\alpha$ is non-trivial, then there must exist a vertex $v \in \Sigma^\star$ such that $\alpha \cdot v \neq v$.
	Thus, we nondeterministically choose such a witness from $\mathcal{R} = \Sigma^\star \mathfrak{t}$ and store it on the check-stack, where the letter $\mathfrak{t}$ represents the top of the check-stack.
	
	From the start state, $q_0$, $\mathcal{M}$ will copy the contents of the check-stack onto the pushdown, then enter the state $q_\mathrm{comp} \in Q$.
	Formally, this will be achieved by adding the transitions (for each $a \in \Sigma$):
	\[
	(
		(q_0,\varepsilon,(\mathfrak{b},\mathfrak{b})),
		(q_0,\mathfrak{t}\mathfrak{b})
	),\,
	(
		(q_0,\varepsilon,(a,\mathfrak{t})),
		(q_0,\mathfrak{t}a)
	),\,
	(
		(q_0,\varepsilon,(\mathfrak{t},\mathfrak{t})),
		(q_\mathrm{comp},\mathfrak{t})
	).
	\]
	
	This stage concludes with $\mathcal{M}$ in state $q_{\mathrm{comp}}$, and the read-head pointing to $(\mathfrak t, \mathfrak t)$. 
	Note that whenever the machine is in state $q_{\mathrm{comp}}$ and $\alpha_1 \alpha_2 \cdots \alpha_k$ has been read from input, then the contents of pushdown will represent the permuted vertex $(\alpha_1 \alpha_2 \cdots \alpha_k) \cdot v$.
	Thus, the two stacks are initially the same as no input has been read and thus no group action has been simulated.
	In Stages~2a and 2b, only the height of the check-stack is impotant, that is, the exact contents of the check-stack will become relevant in Stage~3.
	
	\vspace*{.8em}\noindent
	{\itshape Stage~2a: reading a finitary automorphism $\phi \in Y\cap\mathrm{Fin}(\mathcal{T}_d)$.}
	\vspace*{.25em}
	
	\noindent
	By definition, there exists some $k_\phi = \mathrm{depth}(\phi) \in \mathbb{N}$ such that $\left.\phi\right\vert_u = 1$ for each $u \in \Sigma^\star$ for which $\vert u \vert \geq k_\phi$.
	Thus, given a vertex $v = v_1 v_2 \cdots v_m \in \Sigma^\star$, we have
	\[
		\phi(v)
		=
		\phi(v_1 v_2 \cdots v_{k_\phi})
		\ 
		v_{(k_\phi+1)}
		\cdots
		v_m.
	\]
	
	Given that $\mathcal{M}$ is in state $q_\mathrm{comp}$ with $\mathfrak{t} v_1 v_2 \cdots v_m \mathfrak{b}$ on its pushdown, we will read $\phi$ from input, move to state $q_{\phi,\varepsilon}$ and pop the $\mathfrak{t}$;
	we will then pop the next $k_\phi$ (or fewer if $m < k_\phi$) letters off the pushdown, and as we are popping these letters we visit the sequence of states $q_{\phi,v_1}$, $q_{\phi,v_1 v_2}$, \dots, $q_{\phi,v_1 v_2 \cdots v_{k_\phi}}$.
	From the final state in this sequence, we then push $\mathfrak t\phi(v_1\cdots v_{k_\phi})$ onto the pushdown, and return to the state $q_{\mathrm{comp}}$.
	
	Formally, for letters $a,b \in \Sigma$, $\phi \in Y\cap\mathrm{Fin}(\mathcal{T}_d)$, and vertices $u,w\in\Sigma^\star$ where $ |u|< k_\phi$ and $ |w|=k_\phi$, we have the transitions
	\[
		(
			(q_{\mathrm{comp}}, \phi, (\mathfrak{t},\mathfrak{t})),
			(q_{\phi,\varepsilon}, \varepsilon)
		), \ 
		(
			(q_{\phi,u}, \varepsilon, (a,b)),
			(q_{\phi, ub}, \varepsilon)
		),
	\]
	\[
		(
			(q_{\phi,  w}, \varepsilon, (\varepsilon,\varepsilon)),
			(q_\mathrm{comp}, \mathfrak{t}\phi(w))
		)
	\]
	for the case where $m > k_\phi$, and
	\[
		(
			(q_{\phi, u}, \varepsilon, (\mathfrak{b},\mathfrak{b})),
			(q_\mathrm{comp}, \mathfrak{t}\phi(u)\mathfrak{b})
		)
	\]
	for the case where $m \leq k_\phi$.
	Notice that we have finitely many states and transitions  since $Y,$ $\Sigma$ and each $k_\phi$ is finite.

	\vspace*{.8em}\noindent
	{\itshape Stage~2b: reading a directed automorphism $\delta \in Y\cap\mathrm{Dir}(\mathcal{T}_d)$.}
	\vspace*{.25em}
	
	\noindent
	By \cref{lemma:spine is eventualy periodic}, there exists some $\iota = \iota_1 \iota_2 \cdots \iota_s \in \Sigma^\star$ and $\pi = \pi_1 \pi_2 \cdots \pi_t \in \Sigma^\star$ such that
	$
		\mathrm{spine}(\delta)
		=
		\iota \, \pi^\omega
	$
	and
	\[
		\delta(\iota\pi^\omega)
		=
		I_1 I_2 \cdots I_s
		\,
		\left( \Pi_1 \Pi_2 \cdots \Pi_t \right)^\omega
	\]
	where
	\[
		I_i
		=
		\left. \delta \right\vert_{\iota_1 \iota_2 \cdots \iota_{i-1}} (\iota_i)
		\qquad
		\text{and}
		\qquad
		\Pi_j
		=
		\left. \delta \right\vert_{\iota \, \pi_1 \pi_2 \cdots \pi_{j-1}}\!(\pi_j).
	\]
	
	Given some vertex $v = v_1 v_2 \cdots v_m \in \Sigma^\star$, let $\ell \in \mathbb{N}$ be largest such that $p = v_1 v_2 \cdots v_\ell$ is a prefix of the sequence $\iota\pi^\omega = \mathrm{spine}(\delta)$.
	Then by  definition of directed automorphism, $\delta' = \delta\vert_p$ is directed and $\phi = \delta\vert_a$, where $a = v_{\ell}$, is finitary.
	Then, either $p = \iota_1 \iota_2 \cdots \iota_\ell$ and 
	\[
		\delta(u)
		=
		(I_1 I_2 \cdots I_\ell)
		\ 
		\delta'(a)
		\ 
		\phi(v_{\ell+2} v_{\ell+3} \cdots v_m),
	\]
	or $p = \iota\pi^k\pi_1\pi_2\cdots\pi_j$, with $\ell = |\iota|+k\cdot|\pi| + j$, and 
	\[
		\delta(u)
		=
		(I_1 I_2 \cdots I_s)
		\,
		(\Pi_1 \Pi_2 \cdots \Pi_t)^k
		\,
		(\Pi_1 \Pi_2 \cdots \Pi_j)
		\ 
		\delta'(a)
		\ 
		\phi( v_{\ell+2} v_{\ell+3} \cdots v_m).
	\]
	
	Hence, from state $q_\mathrm{comp}$ with $\mathfrak{t} v_1 v_2 \cdots v_m \mathfrak{b}$ on its pushdown, $\mathcal M$ reads $\delta$ from input, moves to state $q_{\delta,\iota,0}$ and pops the $\mathfrak{t}$;
	it then pops $pa$ off the pushdown, using states to remember the letter $a$ and the part of the prefix to which the final letter of $p$ belongs (i.e.\@ $\iota_i$ or $\pi_j$).
	From here, $\mathcal{M}$ performs the finitary automorphism $\phi$ on the remainder of the pushdown (using the same construction as Stage~2a), then, in a sequence of transitions, pushes $\mathfrak{t}\delta(p)\delta'(a)$ and returns to state $q_\mathrm{comp}$.
	The key idea here is that, using only the knowledge of the letter $a$, the part of $\iota$ or $\pi$ to which the final letter of $p$ belongs, and the height of the check-stack, that $\mathcal M$ is able to recover $\delta(p)\delta'(a)$.
	
	We now give the details of the states and transitions involved in this stage of the construction.
	
	We have states $q_{\delta,\iota,i}$ and $q_{\delta,\pi,j}$ with $0 \leq i \leq \vert \iota \vert$, $1 \leq j \leq \vert \pi \vert$; where $q_{\delta,\iota,i}$ represents that the word $\iota_1 \iota_2 \cdots \iota_i$ has been popped off the pushdown, and $q_{\delta,\pi,j}$ represents that a word $\iota\pi^k\pi_1\pi_2\cdots \pi_j$ for some $k \in \mathbb{N}$ has been popped of the pushdown.
	Thus, we begin with the transition
	\[
		(
			(q_\mathrm{comp}, \delta, (\mathfrak{t},\mathfrak{t})),
			(q_{\delta,\iota,0},\varepsilon)
		),
	\]
	then for each $i,j\in \mathbb{N}$, $a \in \Sigma$ with $0 \leq i < \vert \iota \vert$ and $1 \leq j < \vert \pi \vert$, we have transitions
	\begin{align*}
	  (
		  (q_{\delta,\iota,i}, \varepsilon, (a,\iota_{i+1})),
		  (q_{\delta,\iota,(i+1)},\varepsilon)
	  ),&\ \ 
	  (
		  (q_{\delta,\iota,|\iota|}, \varepsilon, (a,\pi_1)),
		  (q_{\delta,\pi,1},\varepsilon)
	  ),
	  \\
	  (
		  (q_{\delta,\pi,j}, \varepsilon, (a,\pi_{j+1})),
		  (q_{\delta,\pi,(j+1)},\varepsilon)
	  ),&\ \ 
	  (
		  (q_{\delta,\pi,|\pi|}, \varepsilon, (a,\pi_1)),
		  (q_{\delta,\pi,1},\varepsilon)
	  )
	\end{align*}
	to consume the prefix $p$.
	
	After this,  $\mathcal{M}$ will either be at the bottom of its stacks, or its read-head will see a letter on the pushdown that is not the next letter in the spine of $\delta$.
	Thus, for each $i,j \in \mathbb{N}$ with $0 \leq i \leq \vert \iota \vert$ and $1 \leq j \leq |\pi|$ we have states $q_{\delta,\iota,i,a}$ and $q_{\delta,\pi,j,a}$;
	and for each $b \in \Sigma$ we have transitions
	\[
		(
			(q_{\delta,\iota,i}, \varepsilon, (b,a)),
			(q_{\delta,\iota,i,a},\varepsilon)
		)
	\]
	where $a \neq \iota_{i+1}$ when $i < |\iota|$ and $a \neq \pi_1$ otherwise, and
	\[
		(
			(q_{\delta,\pi,j}, \varepsilon, (b,a)),
			(q_{\delta,\pi,j,a},\varepsilon)
		)
	\]
	where $a \neq \pi_{j+1}$ when $j < |\pi|$ and $a \neq \pi_1$ otherwise.
	
	Hence, after these transitions, $\mathcal{M}$ has consumed $pa$ from its pushdown and will either be at the bottom of its stacks in some state $q_{\delta,\iota,i}$ or $q_{\delta,\pi,j}$; or will be in some state $q_{\delta,\iota,i,a}$ or $q_{\delta,\pi,j,a}$.
	Note here that, if $\mathcal{M}$ is in the state $q_{\delta,\iota,i,a}$ or $q_{\delta,\pi,j,a}$, then from \cref{lemma:spine is eventualy periodic} we know $\delta' = \delta\vert_{p}$ is equivalent to $\delta\vert_{\iota_1 \iota_2\cdots \iota_i}$ or $\delta\vert_{\iota \pi_1 \pi_2 \cdots \pi_j}$, respectively; and further, we know the finitary automorphism $\phi = \delta\vert_{pa} = \delta'\vert_a$.
	
	Thus, for each state $q_{\delta,\iota,i,a}$ and $q_{\delta,\pi,a}$ we will follow a similar construction to Stage~2a, to perform the finitary automorphism $\phi$ to the remaining letters on the pushdown, then push $\delta'(a)$ and return to the state $r_{\delta,\iota,i}$ or $r_{\delta,\pi,j}$, respectively.
	For the case where $\mathcal{M}$ is at the bottom of its stacks we have transitions
	\[
	  (
		  (q_{\delta,\iota,i}, \varepsilon, (\mathfrak{b},\mathfrak{b})),
		  (r_{\delta,\iota,i}, \mathfrak{b})
	  ),\ \ 
	  (
		  (q_{\delta,\pi,i}, \varepsilon, (\mathfrak{b},\mathfrak{b})),
		  (r_{\delta,\pi,i}, \mathfrak{b})
	  )
	\]
	with $0 \leq i \leq |\iota|$, $1 \leq j \leq |\pi|$.
	
	Thus, after following these transitions, $\mathcal{M}$ is in some state $r_{\delta,\iota,i}$ or $r_{\delta,\pi,j}$ and all that remains is for $\mathcal{M}$ to push $\delta(p)$ with $p = \iota_1 \iota_2\cdots \iota_i$ or $p = \iota\pi^k\pi_1 \pi_2\cdots \pi_k$, respectively, onto its pushdown.
	Thus, for each $i,j\in \mathbb{N}$ with $0 \leq i \leq |\iota|$ and $1 \leq j \leq |\pi|$, we have transitions
	\[
		(
			(r_{\delta,\pi,i}, \varepsilon,(\varepsilon,\varepsilon)),
			(q_\mathrm{comp}, \mathfrak{t} I_1 I_2 \cdots I_i)
		),
		\ \ 
		(
			(r_{\delta,\pi,j}, \varepsilon,(\varepsilon,\varepsilon)),
			(r_{\delta,\pi}, \Pi_1 \Pi_2 \cdots \Pi_j)
		)
	\]
	where from the state $r_{\delta,\pi}$, through a sequence of transitions, $\mathcal{M}$ will push the remaining $I\Pi^k$ onto the pushdown.
	In particular, we have transitions
	\[
		(
			(r_{\delta,\pi}, \varepsilon,(\varepsilon,\varepsilon)),
			(r_{\delta,\pi},\Pi)
		),
		\ \ 
		(
			(r_{\delta,\pi}, \varepsilon,(\varepsilon,\varepsilon)),
			(q_\mathrm{comp},\mathfrak{t}I)
		),
	\]
	so that $\mathcal{M}$ can nondeterministically push some number of $\Pi$'s followed by $\mathfrak{t}I$  before it finishes this stage of the computation.
	We can assume that the machine pushes the correct number of $\Pi$'s onto its pushdown as otherwise it will not see $\mathfrak{t}$ on its check-stack while in state $q_\mathrm{comp}$ and thus would not be able to continue with its computation, as every subsequent stage (2a,2b,3) of the computation begins with the read-head pointing to $\mathfrak{t}$ on both stacks.
	
	Once again it is clear that this stage of the construction requires only finitely many states and transitions.
	
	\vspace*{.8em}\noindent
	{\itshape Stage~3: checking that the action is non-trivial.}
	\vspace*{.25em}
	
	\noindent
	At the beginning of this stage, the contents of the check-stack represent the chosen witness, $v$, and the contents of the pushdown represent the action of the input word, $\alpha$, on the witness, i.e., $\alpha \cdot v$.
	
	In this stage $\mathcal{M}$  checks if the contents of its check-stack and pushdown differ.
	Formally, we have states $q_\mathrm{accept}$ and  $q_{\mathrm{check}}$, with $q_\mathrm{accept}$ accepting;
	 for each $a \in \Sigma$, we have transitions
	\[
		(
			(q_\mathrm{comp}, \varepsilon, (\mathfrak{t},\mathfrak{t})),
			(q_\mathrm{check}, \varepsilon)
		),
		\ \ 
		(
			(q_\mathrm{check}, \varepsilon, (a,a)),
			(q_\mathrm{check}, \varepsilon)
		)
	\]
	to pop identical entries of the pushdown; and for each $(a,b) \in \Sigma \times \Sigma$ with $a \neq b$ we have a transition
	\[
		(
		(q_\mathrm{check},\varepsilon,(a,b)),
		(q_\mathrm{accept},\varepsilon)
		)
	\]
	to accept if the stacks differ by a letter.
	
	Observe that if the two stacks are identical, then there is no path to the accepting state, $q_\mathrm{accept}$, and thus $\mathcal{M}$ will reject.
	Notice also that by definition of cspd automata, if $\mathcal M$ moves into  $q_{\mathrm{check}}$ before all input has been read, then $\mathcal{M}$ will not accept, i.e., an accepting state is only effective if all input is consumed.
	
	\vspace*{.8em}\noindent
	{\itshape Soundness and Completeness.}
	\vspace*{.25em}

	\noindent
	If $\alpha$ is non-trivial, then there is a vertex $v \in \Sigma^\star$ such that $\alpha \cdot v \neq v$, which $\mathcal{M}$ can nondeterministically choose to write on its check-stack and thus accept $\alpha$.
	If $\alpha$ is trivial, then $\alpha \cdot v = v$ for each vertex $v \in \Sigma^\star$, and there is no choice of checking stack for which $\mathcal{M}$ will accept, so $\mathcal{M}$ will reject.
	
	Thus, $\mathcal{M}$ accepts a word if and only if it is in $\mathrm{coW}(H,Y)$.
	Hence, the co-word problem $\mathrm{coW}(H,Y)$ is ET0L,  completing our proof.
\end{proof}

\bibliographystyle{plain}
\bibliography{references}

\addresseshere

\clearpage
\appendix

\section{Equivalence of CSPD and ET0L}
\label{sec:machine equivalence}

In this appendix we will provide a self-contained proof of the equivalence between the class of ET0L languages and the class of languages recognised by cspd automata.

\subsection{Notation}
\label{subsec:appendix:notation}

It will be convenient to define a non-terminal  $\mathfrak{d}$ which we call a \emph{dead-end symbol}.
Given a grammar with a dead-end symbol, we will demand that $\mathfrak{d}$ is not a terminal and that each table can only map $\mathfrak{d}$ to itself, i.e., $\mathfrak{d} \to \mathfrak{d}$.
Thus, if a table induces a letter $\mathfrak{d}$, then there is no way to remove it to generate a word in the associated language.

For simplicity when presenting tables, if a replacement is not specified for a particular variable $X$, then it should be assumed that the replacement rule $X \to X$ is in the table.

\begin{lemma}[Christensen \cite{christensen1974}]
	\label{lemma:et0l rhs}
	The class of ET0L grammars does not gain any expressive power if each replacement rules is of the form
	$
	\tau : X \to L_{X,\tau}
	$
	where each $L_{X,\tau}$ is an ET0L language;
	meaning that $\tau$ replaces instances of the variable $X$ with words from $L_{X,\tau}$.
\end{lemma}

\begin{proof}
	
	Let $G = (\Sigma, V, T, \mathcal{R}, S)$ be a grammar in this extended form, that is, where each replacement rule maps $X$ into any word in some ET0L language $L_{X,\tau}$. 
	Assume without loss of generality that every terminal is also a non-terminal, i.e., $\Sigma \subseteq V$
	(this is done by  first adding replacement rules $\tau:a \to a$ for each table $\tau \in T$ and each $a \in \Sigma \setminus V$; then we add the letters of $\Sigma$ to $V$.
	It is clear that this modified grammar generates the same language).

	For each language $L_{X,\tau}$ in the grammar $G$, let
	\[
		G_{X,\tau}
		=
		(
		\Sigma_{X,\tau},
		V_{X,\tau},
		T_{X,\tau},
		\mathcal{R}_{X,\tau},
		S_{X,\tau}
	)
	\]
	be an ET0L grammar such that $L_{X,\tau} = L(G_{X,\tau})$.
	Notice here that $\Sigma_{X,\tau}$ must be a subset of $V$ such that the language $L_{X,\tau}$ generates words in $V^\star$.
	
	For each $X \in V$, we define two disjoint copies denoted as $X^{(1)}$ and $X^{(2)}$;
	and we demand that these copies are disjoint to letters in the alphabets $V_{Y,\tau}$ and $\Sigma_{Y,\tau}$ for each $Y\in V$ and $\tau \in T$.
	
	We define two ET0L tables $\alpha$ and $\kappa$ such that, for each $X \in V$, we have replacement rules $\alpha : X \to X^{(1)}$ and $\kappa: X^{(1)} \to \mathfrak{d}$.
	
	For each table $\tau \in T$ and non-terminal $X \in V$, we define ET0L tables
	\[
		\beta_{X,\tau}: X^{(1)} \to X^{(1)}\ \vert \ S_{X,\tau}
		\ \ \ 
		\text{and}
		\ \ \ 
		\gamma_{X,\tau}:
		\left\{
		\begin{aligned}
		Y &\to Y^{(2)}     && \text{for }Y \in \Sigma_{X,\tau}\subseteq V,\\
		Z &\to \mathfrak{d}&& \text{for }Z \in V_{X,\tau} \setminus \Sigma_{X,\tau}.
		\end{aligned}
		\right.
	\]
	
	Given a $\tau \in T$, it can be seen that $\tau$ is equivalent to the regular expression
	\[
		\tau'
		=
		\alpha
		\left( \beta_{X_1}^\tau \mathcal{R}_{X_1}^\tau \gamma_{X_1}^\tau\right)^\star
		\left( \beta_{X_2}^\tau \mathcal{R}_{X_2}^\tau \gamma_{X_2}^\tau\right)^\star
		\cdots
		\left( \beta_{X_k}^\tau \mathcal{R}_{X_k}^\tau \gamma_{X_k}^\tau\right)^\star
		\kappa
	\]
	where $\{X_1,X_2,\ldots,X_k\}=V$.
	Thus, after replacing each $\tau$ in a regular expression for $\mathcal{R}$ with its corresponding expression $\tau'$, we obtain a regular language which we denote $\mathcal{R}'$.
	Thus, it can be seen that the grammar $G$ is equivalent to an ET0L grammar with rational control given by $\mathcal{R}'$.
\end{proof}

\clearpage
\begin{lemma}
	\label{lemma:restricted form}
	Given a cspd automaton, $\mathcal{M} = (Q, \Sigma, \Gamma, \Delta, \mathfrak{b}, \mathcal{R}, \theta, q_0, F)$, we may assume without loss of generality that:
	
	\begin{enumerate}
		
		\item
		the pushdown is never higher than the check-stack;
		
		\item
		there is only one accepting state, i.e.\@ $\{q_\mathrm{accept}\} = F$;
		
		\item
		the pushdown is empty when $\mathcal{M}$ enters the accepting state $q_\mathrm{accept}$;
		
		\item
		transitions to the accepting state $q_\mathrm{accept}$ do not modify the pushdown;
		
		\item
		each transition to states other than $q_\mathrm{accept}$ either pushes exactly one letter to the pushdown or pops exactly one letter from the pushdown.
		
	\end{enumerate}
\end{lemma}

\begin{proof}
	
	Let  $\mathcal{M} = (Q, \Sigma, \Gamma, \Delta, \mathfrak{b}, \mathcal{R}, \theta, q_0, F)$ be a cspd machine.
	
	For assumption~1,
	 let $N \in \mathbb{N}$ be an upper bound on the number of letter any transition of $\mathcal{M}$ can push.
	That is, $N$ is such that, given any transition
	$
		(
			(q,a,(d,g)),
			(p,b_1 b_2 \cdots b_k)
		)
		\in \theta
	$
	where each $b_j \in \Gamma$, it is the case that $k \leq N$.
	We now add a disjoint letter $\mathfrak{t}$ to the check-stack alphabet $\Delta$. 
	Thus, $\mathcal{M}$ has no available transitions when it sees $\mathfrak{t}$ on its check-stack.
	We thus satisfy assumption~1
	after replacing the regular language of check-stacks with $\mathcal{R}\mathfrak{t}^N$ (where $\mathfrak{t}^N$ is a sequence of $N$ letters $\mathfrak{t}$'s).
	
	For assumptions~2--4
	we introduce states $q_\mathrm{finish}$ and $q_\mathrm{accept}$ disjoint from all other states in $Q$.
	For each $(d,g) \in \Delta \times \Gamma$ and each $q \in F$ we add
	\[
		(
			(q,\varepsilon,(d,g)),
			(q_\mathrm{finish},\varepsilon)
		),
		\ 
		(
			(q_\mathrm{finish},\varepsilon, (d,g)),
			(q_\mathrm{finish},\varepsilon)
		)
	\]
	so that we can empty the pushdown after reaching an accepting state $q\in F$.
	Further, for each state $q \in F$, we add transitions
	\[
		(
			(q,\varepsilon,(\mathfrak{b},\mathfrak{b})),
			(q_\mathrm{accept},\mathfrak{b})
		),
		\ 
		(
			(q_\mathrm{finish},\varepsilon, (\mathfrak{b},\mathfrak{b})),
			(q_\mathrm{accept},\mathfrak{b})
		)
	\]
	so that we can move to state $q_\mathrm{accept}$ once the pushdown has been emptied.
	
	Thus, we now replace the set of accepting states, $F$, with $\{q_\mathrm{accept}\}$ to obtain an equivalent machine that satisfies assumptions~2--4.
	
	For assumption~5, 
	we only need to consider transitions of the form
	\begin{equation}
	\label{eq:form1}
		(
			(p,\alpha,(d,g)),
			(q,a_1 a_2 \cdots a_k)
		)
	\end{equation}
	where $p,q \in Q$ with $q \neq q_\mathrm{accept}$, $(d,g) \in \Delta\times\Gamma\sqcup\{(\mathfrak{b},\mathfrak{b})\}$, $\alpha \in \Sigma^\star$, and each $a_j \in \Gamma\sqcup \{\mathfrak{b}\}$ with either $a_k \neq g$ or $k > 2$.
	
	Given a transition as in (\ref{eq:form1}), we add states $p_{a_1}^q$, $p_{a_1 a_2}^q$, \dots, $p_{a_1 a_2 \cdots a_k}^q$; and
	for each $(b,c) \in \Delta\times\Gamma\sqcup\{(\mathfrak{b},\mathfrak{b})\}$ and $j\in \mathbb{N}$ with $1 < j \leq k$, we add transitions
	\[
		(
			(p^q_{a_1 a_2 \cdots a_j},\varepsilon,(b,c)),
			(p^q_{a_1 a_2 \cdots a_{j-1}},a_{j}c)
		),
		\ 
		(
			(p^q_{a_1},\varepsilon,(b,c)),
			(q,a_{1}c)
		),
	\]
	so that, from state $p^q_{a_1 a_2 \cdots q_j}$, we go through a sequence of transitions which push the word $a_1 a_2 \cdots a_j$ and end in the state $q$.
	Moreover, if $a_k \neq g$ in our given transition, then we add a transition
	\[
		(
			(p,\alpha, (d,g)),
			(p^q_{a_1a_2\cdots a_k}, \varepsilon)
		),
	\]
	otherwise we add a transition
	\[
		(
			(p,\alpha, (d,g)),
			(p^q_{a_1a_2\cdots a_{k-2}}, a_{k-1} g)
		).
	\]
	
	Notice that with the addition of these states and transitions, we can remove all transitions as in (\ref{eq:form1}) and still recognise the same language.
	
	This completes the proof.
\end{proof}

\subsection{Proof of Equivalence}
\label{subsec:appendix:proof of equivalence}

In this section we prove 
\begin{theorem}[van Leeuwen \cite{leeuwen1976}]
	\label{thm:main equivalence theorem}
	The class of ET0L languages is equivalent to the class of languages recognised by cspd automata.
\end{theorem}
The proof  is divided into two propositions.

\begin{proposition}
	\label{prop:grammar to cspd machine}
	Given an ET0L language, there is a cspd automaton which recognises it.
\end{proposition}

\begin{proof}
	
	Let $L$ be a given ET0L language with grammar $G = (\Sigma, V, T, \mathcal{R}, S)$, i.e., $L = L(G)$.
	Thus, in this proof we construct a cspd machine $\mathcal{M}$ which accepts precisely the language $L$ by emulating derivations of words with respect to $G$.
	
	Let $w \in L$.
	Then, by the definition of an ET0L grammar, there exists a sequence of tables $\alpha = \alpha_1 \alpha_2 \cdots a_k \in \mathcal{R}$ for which $S \to^\alpha w$.
	Thus, the idea of this construction is that $\mathcal{M}$ accepts the word $w$ by choosing its check-stack to represent such a sequence $\alpha$, then $\mathcal{M}$ emulates a derivation of $w$ from $S$ with the use of its pushdown and reference to the check-stack.
	
	We now give a description of this construction.
	We begin by choosing the alphabets and states for $\mathcal{M}$ as follows.
	
	The input alphabet of $\mathcal{M}$ is given by $\Sigma$.
	The alphabet of the check-stack is given by $\Delta = T \sqcup \{\mathfrak{t}\}$ where $\mathfrak{t}$ is used to denote the top of the check-stack.
	Further, the regular language of allowed check-stacks is given by $\mathcal{R}\mathfrak{t}$ where $\mathcal{R}$ is the rational control of the grammar $G$.
	
	The alphabet for the pushdown, $\Gamma$, will include letters $\llbracket S \rrbracket$ and $\llbracket \varepsilon \rrbracket$ to denote the starting symbol and empty word, respectively; and for each table $\tau \in T$ and each replacement rule $\tau : A \to B_1 B_2 \cdots B_\ell$ with each $B_j \in V \cup \Sigma$, and for each $k \in \mathbb{N}$ with $1 \leq k \leq \ell$, we have $\llbracket B_k B_{k+1} \cdots B_\ell \rrbracket$ as a distinct symbol of $\Gamma$.
	For example, if $T = \{\alpha, \beta, \gamma\}$ where
	\[
	\alpha \, : \,
	S \to ABC \ \vert \ BCB
	\qquad
	\beta \, : \,
	\left\{
	\begin{aligned}
		A &\to BA \ \vert \ B \\
		B &\to B\ \vert\ \varepsilon \\
		C &\to B
	\end{aligned}
	\right.
	\qquad
	\gamma \, : \,
	\left\{
	\begin{aligned}
		A &\to a \\
		B &\to bb\\
		C &\to c
	\end{aligned}
	\right.,
	\]
	then the corresponding pushdown alphabet is given by
	\begin{multline*}
	\Gamma
	=
	\{ \llbracket S \rrbracket, \llbracket \varepsilon \rrbracket \}
	\cup
	\underbrace{
		\left\{
		\llbracket ABC \rrbracket,
		\llbracket BC \rrbracket,
		\llbracket C \rrbracket,
		\llbracket BCB \rrbracket,
		\llbracket CB \rrbracket,
		\llbracket B \rrbracket
		\right\}
	}_{\text{rules in }\alpha}
	\\
	\cup
	\underbrace{
		\left\{
		\llbracket BA \rrbracket,
		\llbracket A \rrbracket,
		\llbracket B \rrbracket
		\right\}
	}_{\text{rules in }\beta}
	\cup
	\underbrace{
		\left\{
		\llbracket a \rrbracket,
		\llbracket bb \rrbracket,
		\llbracket b \rrbracket,
		\llbracket c \rrbracket
		\right\}
	}_{\text{rules in }\gamma}.
	\end{multline*}
	Notice that the pushdown alphabet $\Gamma$ is finite as an ET0L grammar can have only finitely many replacement rules.
	
	The machine $\mathcal{M}$ has three states $\{ q_0, q_\mathrm{apply}, q_\mathrm{accept}\} = Q$ where $q_0$ is the start state and $q_\mathrm{accept}$ is the only accepting state.
	The idea of state $q_\mathrm{apply}$ is that its transitions to itself emulate an application of a table, which it sees on the check-stack, to a non-terminal, which it sees on the pushdown.
	
	Now that we have chosen our alphabets and states, we are ready to describe the transition relations, $\theta$, of $\mathcal{M}$.
	
	To begin a computation we have the transition
	\[
		(
			(q_0,\varepsilon,(\mathfrak{b},\mathfrak{b})),
			(q_\mathrm{apply}, \llbracket S \rrbracket \mathfrak{b})
		)
	\]
	which pushes the start symbol of the grammar onto the pushdown (see \cref{fig:machine beginning}).
	In the remainder of this proof, we will ensure that $\mathcal{M}$ is only able to empty its pushdown by emulating a derivation of its input word with respect to the grammar $G$.
	Thus, we have the transition
	\[
		(
			(q_\mathrm{apply},\varepsilon,(\mathfrak{b},\mathfrak{b})),
			(q_\mathrm{accept},\mathfrak{b})
		),
	\]
	to accept when the pushdown is emptied.
	
	\begin{figure}[!ht]
		\centering
		\includegraphics[width=.3\linewidth]{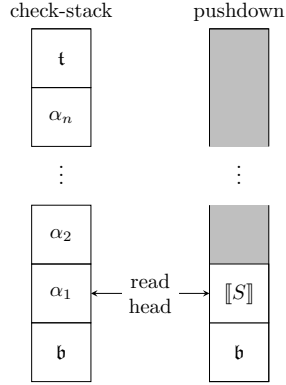}
		\caption{The stack configuration when the machine first enters state $q_\mathrm{apply}$.}
		\label{fig:machine beginning}
	\end{figure}
	
	We will now describe how the transitions from $q_\mathrm{apply}$ performs a derivation in the grammar $G$.
	
	Suppose that $\mathcal{M}$ is in state $q_\mathrm{apply}$, then the remaining transitions can be separated into the following three cases.
	
	\begin{enumerate}
		\item \textit{Applying a table.}
		Suppose that the read-head of $\mathcal{M}$ sees
		$(\tau, \llbracket A_1 A_2 \cdots A_m \rrbracket)$
		where $\tau$ is a table of the grammar and $m \geq 1$.
		Then, we want $\mathcal{M}$ to apply the table $\tau$ to the word $A_1 A_2 \cdots A_m$.
		We do this by applying $\tau$ from left-to-right, i.e., for each $B_1 B_2 \cdots B_k \in (V \cup \Sigma)^\star$ such that $A_1 \to^\tau B_1 B_2 \cdots B_k$ we have transitions
		\[
			(
				(q_\mathrm{apply}, \varepsilon, (\tau, \llbracket A_1 A_2 \cdots A_m \rrbracket)),
				(q_\mathrm{apply},
				\llbracket B_1 B_2 \cdots B_k \rrbracket
				\llbracket A_2 \cdots A_m \rrbracket
				)
			)
		\]
		to expand the letter $A_1$ to a nondeterministic choice of sequence $B_1 B_2 \cdots B_k$.
		See \cref{fig:cspd expanding letter} for an example fo this expansion.
		
		\item \textit{Empty word.}
		Suppose that the read-head of $\mathcal{M}$ sees
		$(\tau, \llbracket \varepsilon \rrbracket)$
		where $\tau$ is a table.
		Then, since there is no way to expand $\varepsilon$ any further, we pop this letter from the pushdown, i.e., for each table $\tau \in T$ we have a transition
		\[
			(
				(q_\mathrm{apply}, \varepsilon, (\tau, \llbracket \varepsilon \rrbracket)),
				(q_\mathrm{apply}, \varepsilon)
			).
		\]
		
		\item \textit{Finished applying tables.}
		Suppose that $\mathcal{M}$ is at the top of its check-stack, i.e., its read-head sees $(\mathfrak{t}, \llbracket A_1 A_2 \cdots A_m \rrbracket)$ where $m=0$ corresponds to the read-head seeing $(\mathfrak{t},\llbracket \varepsilon \rrbracket)$.
		Then, we have no further tables to apply to $A_1 A_2 \cdots A_m$.
		Thus, each $A_j$ must be a terminal letter of $\Sigma$ and we must see $A_1 A_2 \cdots A_m$ on the input tape.
		Thus, for each $\llbracket a_1 a_2 \cdots a_m \rrbracket \in \Gamma$ with each $a_j \in \Sigma$, we have a transition
		\[
		(
		(q_\mathrm{apply}, a_1 a_2 \cdots a_m, (\mathfrak{t}, \llbracket a_1 a_2 \cdots a_m \rrbracket)),
		(q_\mathrm{apply}, \varepsilon)
		).
		\]
		Notice here that, if the letter on the pushdown does not correspond to a word in $\Sigma^\star$, then we have no path to $q_\mathrm{accept}$ and thus we reject.
		
	\end{enumerate}

	\begin{figure}[!ht]
		\centering
		\includegraphics[width=.4\linewidth]{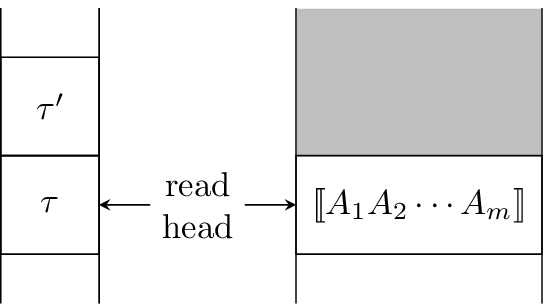}
		\hspace*{.05\linewidth}
		\includegraphics[width=.4\linewidth]{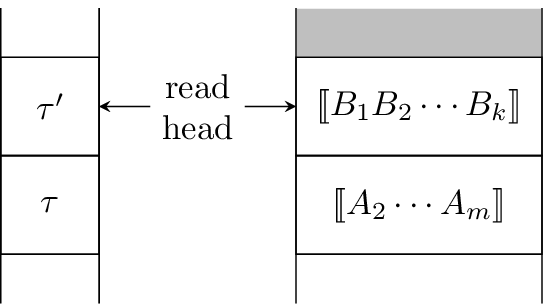}
		\caption{Expanding the letter $A_1$ with respect to the table $\tau$.}
		\label{fig:cspd expanding letter}
	\end{figure}
	
	\vspace*{.5em}\noindent
	{\itshape Soundness and Completeness.}
	\vspace*{.25em}
	
	\noindent
	Suppose that $\mathcal{M}$ is given a word $w \in L$ on its input tape.
	Then, there must exist some $\alpha \in \mathcal{R}$ such that $S \to^\alpha w$ in the grammar $G$.
	Thus, $\mathcal{M}$ can nondeterministically choose a check-stack of $\alpha\mathfrak{t}$ and emulate a derivation of $w$ from $S$ as previously described.
	Hence, $\mathcal{M}$ will accept any word from $L$.
	
	Suppose that $\mathcal{M}$ accepts a given word $w \in L$ with a check-stack of $\alpha\mathfrak{t}$.
	Then, by following the previous construction, it can be seen that we can recover a derivation $S \to^\alpha w$.
	Hence, $\mathcal{M}$ can only accept words in $L$.
	
	Therefore, $\mathcal{M}$ accepts a given word if and only if it is in the language $L$.
	Thus completes the proof.
\end{proof}

\begin{proposition}
	\label{prop:cspd machine to ET0L grammar}
	Any language recognised by a cspd automaton is ET0L.
\end{proposition}

\begin{proof}
	Let $\mathcal{M} = (Q, \Sigma, \Gamma, \Delta, \mathfrak{b}, \mathcal{R}, \theta, q_0, F)$ be a given cspd automaton, where
	we will assume without loss of generality that $\mathcal{M}$ satisfies \cref{lemma:restricted form}.
	
	We will construct a grammar $G = (\Sigma,V,T,\mathcal{R}',S)$ as in \cref{lemma:et0l rhs}, which generates precisely the language recognised by $\mathcal{M}$.
	
	Considering assumptions~2--4
	from \cref{lemma:restricted form}, a plot of the height of the pushdown during a successful computation of $\mathcal{M}$ (i.e.\@ one that leads to the accepting state) will resemble a Dyck path; that is, the non-negative height of the pushdown is zero at the beginning and end of such a computation.
	
	For each pair of states $p,q \in Q$ and each pushdown letter $g \in \Gamma$, the grammar $G$ has a non-terminal letter $A^g_{p,q}$.
	The non-terminal $A^g_{p,q}$ corresponds to the situation where $\mathcal{M}$ has just pushed $g$ onto its pushdown on a transition to the state $p$; and that when $\mathcal{M}$ pops this $g$, it will do so on a transition to the state $q$.
	Further, $G$ has a non-terminal $A^\mathfrak{b}_{q_0,q_\mathrm{accept}}$ which corresponds to any path from the initial configuration to the accepting state.
	(See \cref{fig:cspd dyck path}.)
	Thus, the starting symbol of $G$ will be given by $S = A^\mathfrak{b}_{q_0,q_\mathrm{accept}}$.
	
	For each letter $c \in \Delta \sqcup \{\mathfrak{b}\}$ on the check-stack, we have a table $\tau_c \in T$ in the grammar $G$.
	Moreover, by taking a regular expression for the language $\mathfrak{b}\mathcal{R}$ and replacing each instance of $c \in \Delta \sqcup \{\mathfrak{b}\}$ with its corresponding table $\tau_c$, we obtain the rational control $\mathcal{R}'$ of the grammar $G$.
	
	\begin{figure}[!ht]
		\centering
		\includegraphics[width=.8\linewidth]{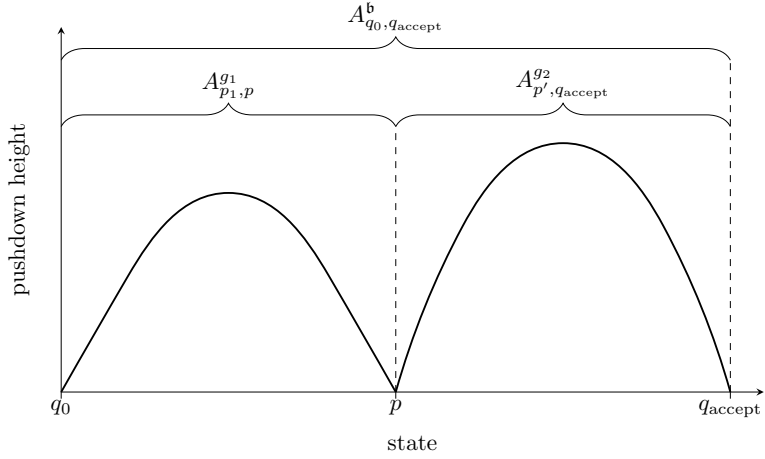}
		\caption{The height of the pushdown during an example computation.}
		\label{fig:cspd dyck path}
	\end{figure}

	Thus, in the remainder of this proof we describe the tables $\tau_c$, and the way in which they emulate a computation of $\mathcal{M}$.
	Note that when describing these tables we make use of the notation introduced in \cref{lemma:et0l rhs};
	in particular, we will use replacements with regular languages on their right-hand sides.
	
	For each $p,q \in Q$ and $(c,b) \in (\Delta\times \Gamma)\sqcup\{(\mathfrak{b},\mathfrak{b})\}$, let $\mathcal{F}_{p,q}^{(c,b)}$ be a finite-state automaton.
	The idea here is that, with $\mathcal{L}_{p,q}^{(c,b)}$ as the regular language accepted by $\mathcal{F}_{p,q}^{(c,b)}$, we will have the replacement rule $\tau_c : A_{p,q}^{b} \to \mathcal{L}_{p,q}^{(c,b)}$.
	
	The states of each $\mathcal{F}_{p,q}^{(c,b)}$ include all the states of $\mathcal{M}$, and an additional disjoint state $\lambda$.
	We denote these states as $Q' = Q \sqcup \{\lambda\}$ where $Q$ are states of $\mathcal{M}$.
	The state $\lambda$ is the only accepting state in each $\mathcal{F}_{p,q}^{(c,b)}$.
	
	Let some $\mathcal{F}_{p,q}^{(c,b)}$ be given.
	Then, a given state $r \in Q \subseteq Q'$ of $\mathcal{F}_{p,q}^{(c,b)}$ corresponds to the situation where $\mathcal{M}$ is in state $r$ and its read-head sees $(c,b)$.
	Thus, the start state of $\mathcal{F}_{p,q}^{(c,b)}$ is given by $p\in Q'$.
	Further, the accepting state $\lambda$ corresponds to the configuration of $\mathcal{M}$ immediately after following a path described by $A_{p,q}^b$.
	
	With a finite-state automaton $\mathcal{F}_{p,q}^{(c,b)}$ given, we now describe its transitions.
	Suppose that $\mathcal{M}$ is in the state $r \in Q$ and its read-head sees $(c,b)$,
	then $\mathcal{M}$ can either push some letter $x \in \Delta$ with a transition of the form
	\begin{equation}
	\label{eq:transition 1}
		(
			(r, \alpha_1 \alpha_2 \cdots \alpha_k, (c,b)),
			(s,xb)
		)
	\end{equation}
	then follow a path described by $A_{s,t}^{x}$ for some state $t\in Q$;
	or $\mathcal{M}$ can complete a path described by $A_{r,q}^b$ with a transition of the form
	\begin{equation}
	\label{eq:transition 2}
		(
			(r, \alpha_1 \alpha_2 \cdots \alpha_k, (c,b)),
			(q,\varepsilon)
		)
		\quad
		\text{or}
		\quad
		(
			(r, \alpha_1 \alpha_2 \cdots \alpha_k, (\mathfrak{b},\mathfrak{b})),
			(q, \mathfrak{b})
		)
	\end{equation}
	depending on whether $q$ is the accepting state $q_\mathrm{accept}$ (see \cref{lemma:restricted form}).
	
	Thus, for each transition in $\mathcal{M}$ of form (\ref{eq:transition 1}), and each state $t\in Q \subseteq Q'$, we have a transition in $\mathcal{F}_{p,q}^{(c,b)}$ from state $r$ to $t$ labelled $\alpha_1 \alpha_2 \cdots \alpha_k A_{s,t}^{x}$.
	
	Further, for each transition of form (\ref{eq:transition 2}), we have a transition in $\mathcal{F}_{p,q}^{(c,b)}$ from state $r$ to $\lambda$ labelled $\alpha_1 \alpha_2 \cdots \alpha_k$.
	
	For each check-stack letter $c \in \Delta \sqcup\{\mathfrak{b}\}$ of $\mathcal{M}$ and non-terminal $A_{p,q}^b$ of $G$, we define the tables $\tau_c$ such that $\tau_c : A_{p,q}^{b} \to \mathcal{L}_{p,q}^{(c,b)}$ where $\mathcal{L}_{p,q}^{(c,b)}$ is the regular language recognised by $\mathcal{F}_{p,q}^{(c,b)}$.
	Since regular language is a subset of ET0L, then, by \cref{lemma:et0l rhs}, the grammar $G$ produces an ET0L language as required.
	
	\vspace*{.5em}\noindent
	{\itshape Soundness and Completeness.}
	\vspace*{.25em}
	
	\noindent
	Suppose that $\mathcal{M}$ is able to accept the word $w \in \Sigma^\star$ with $\alpha \in \mathcal{R}$ chosen as its check-stack.
	Then, by following such a computation to the accepting state, we can construct a derivation $S \to^{\mathfrak{b}\alpha} w$ in the grammar $G$.
	Thus, every word that is accepted by $\mathcal{M}$ is in the language produce by $G$.
	
	Let $w\in L(G)$ be a word produced by the grammar $G$.
	Then, there must exists some sequence of tables $\beta \in \mathcal{R}'$ such that $S \to^{\beta} w$; and thus, for any corresponding derivation in the grammar $G$, and by following the our construction, we can recover a computation of $\mathcal{M}$ which accepts $w$.
	
	Therefore, $G$ generates precisely the language that is recognised by $\mathcal{M}$.
	Thus completes the proof.
\end{proof}

\end{document}